\documentclass[12pt,reqno, a4paper]{amsart}

\usepackage{amsmath,mathrsfs}
\usepackage{mathrsfs}
\usepackage{amssymb}
\usepackage{color}
\usepackage{calligra}
\usepackage{pstricks,pst-node}
\usepackage[latin1]{inputenc}
\usepackage{bbm}

\newcommand\NoBlackBoxes{\global\overfullrule0pt}
\NoBlackBoxes

\newcommand{\ba}{\begin{eqnarray} }
\newcommand{\ea}{\end{eqnarray} }
\newcommand{\N}{\mathbb{N}}
\newcommand{\Ns}{\mathscr{N}}

\makeatletter
\let\serieslogo@\relax
\let\@setcopyright\relax

\newtheorem{definition}{Definition}[section]
\newtheorem{theorem}[definition]{Theorem}
\newtheorem{lemma}[definition]{Lemma}

\newtheorem{proposition}[definition]{Proposition}

\newtheorem{corollary}[definition]{Corollary}

\renewcommand{\P}{{\mathbb{P}}}

\newcommand{\R}{{\mathbb{R}}}

\renewcommand{\epsilon}{\varepsilon}
\renewcommand{\phi}{\varphi}

\numberwithin{equation}{section}

\begin{document}

\setcounter{page}{1}

\title[Multi-group Binary Choice with Social Interaction]{Multi-group Binary Choice with Social Interaction and a random communication structure --  a random graph approach}

\author[Matthias L\"owe]{Matthias L\"owe}
\address[Matthias L\"owe]{Fachbereich Mathematik und Informatik,
Universit\"at M\"unster,
Einsteinstra\ss e 62,
48149 M\"unster,
Germany}

\email[Matthias L\"owe]{maloewe@math.uni-muenster.de}

\author[Kristina Schubert]{Kristina Schubert}
\address[Kristina Schubert]{Fakult\"at f\"ur  Mathematik,
Technische Universit\"at Dortmund,
Vogelpothsweg 87,
44227 Dortmund,
Germany}

\email[Kristina Schubert]{kristina.schubert@tu-dortmund.de}

\author[Franck Vermet]{Franck Vermet}
\address[Franck Vermet]{Laboratoire de Math\'ematiques de Bretagne Occidentale,
Universit\'e de Bretagne Occidentale,
6, Avenue Victor Le Gorgeu,
29238 BREST Cedex 3,
FRANCE}

\email[Franck Vermet]{franck.vermet@univ-brest.fr}




\newcommand{\wlim}{\mathop{\hbox{\rm w-lim}}}
\newcommand{\na}{{\mathbb N}}
\newcommand{\re}{{\mathbb R}}

\newcommand{\vep}{\varepsilon}

\begin{abstract}
We construct and analyze a random graph model for discrete choice with social interaction and several groups of equal size. We concentrate on the case of two groups of equal sizes and we allow the interaction strength within a group to differ from the interaction strength between the two groups.
Given that the resulting graph is sufficiently dense we show that, with probability 1, the average decision in each of the two groups is the same as in the fully connected model. In particular, we show that there is a phase transition: If the interaction among a group and between the groups is strong enough the average decision per group will either be positive or negative and the decision of the two groups will be correlated. We also compute the free energy per particle in our model.
\end{abstract}

\maketitle

\section{Introduction}
As the study of social phenomena, like  decision making processes or voting, has become the subject of various scientific disciplines, a variety of approaches has emerged towards such topics.  Accordingly, different aspects are stressed from an economic and a   sociological point of view:  While the role of individual
preferences (see \cite{becker}) is usually in the focus af economic models, in sociological models individuals are regarded as members of a group and the individual's behaviour is essentially determined by the behaviour of the group (see e.g.~\cite{coleman}, \cite{bourdieu}).

A unifying approach are so called social interaction models. First attempts to use such models go back to Schelling \cite{schelling71}. F\"ollmer \cite{foellmer74} used the theory of Markov random fields from statistical physics to furnish this approach with a rigorous mathematical framework. In the 1990s and early 2000s interacting spins systems were discovered as a model for (mostly binary) discrete choice problems with social interaction, see e.g.~\cite{brockdurlauf}, \cite{EKT03}, \cite{horst}, \cite{contloewe}, \cite{horstscheinkman}, or \cite{kl07}. From this list \cite{brockdurlauf} is particularly interesting for our paper, because it gives a reinterpretation of the Curie-Weiss model from statistical physics in terms of discrete choice models with interactions. Our contribution will be to extend this result to two groups and a random communication structure.

The considerations of decision making in more than one group, where the members of one group interact with one interaction strength, while members of two groups interact with a different strength, led to so-called bipartite Curie-Weiss models, that were analyzed in a statistical mechanics context, see e.g.~see \cite{gallo_contucci_CW}, \cite{gallo_barra_contucci}, \cite{fedele_contucci}, \cite{collet_CW}, or \cite{opoku}. These models were also considered from a statistical perspective recently by Berthet, Rigollet and Srivastava \cite{BRS17_blockmodel} as a version of a statistical block model.
Such block models have been in the center of interest in statistics and probability theory over the past couple of years (see, e.g.~\cite{AL18}, \cite{GMZZ17}, \cite{BRS17_blockmodel}). The statistical interest arises from their relation to graphical models, while from a probabilistic point of view they can be considered to model social interactions with respect to certain decisions, see e.g.~\cite{BGA08}. In this framework a major question is always how to reconstruct the block structure under sparsity assumptions (see e.g.~\cite{BMS13}, \cite{MNS16}, \cite{Bre15}).

In the model from \cite{BRS17_blockmodel}, that has been studied earlier in \cite{gallo_contucci_CW}, \cite{fedele_contucci}, and also later in \cite{werner_curie_weiss}, \cite{LSblock1}, and \cite{werner_curie_weiss2} and
that is interesting both from a probabilistic and a statistical perspective, one partitions the set $\{1, \ldots, N\}$ into a set $S\subset \{1, \ldots, N\}$ and its complement $S^c$. This segmentation induces a
partitioning of the binary hypercube $\{-1,+1\}^N, N\in \N$, the state space of the {\it Ising block model}.
The authors then consider a situation, where the interaction between spins in $S$ resp.~in $S^c$ is stronger than the interaction between spins that belong to different blocks. In \cite{BRS17_blockmodel} the authors describe the statistical mechanics of these models and show how to efficiently reconstruct the blocks $S$ and $S^c$ from observations of the model. In the context of \cite{BRS17_blockmodel} the partitioning is always such that $|S|= N/2$ (for $N$ even) and that there are two blocks only. In some papers that were cited above, e.g.~\cite{fedele_contucci}, a more general set-up with variable block sizes and more than two blocks was considered. In this more general situation it seems non-trivial to give a verifiable condition for the existence of a phase transition and a description of the equilibrium points (note, however, that some results were obtained in \cite{KLSS19}). We will therefore start with the situation described in \cite{BRS17_blockmodel} as a reference model.

The result in \cite{BRS17_blockmodel}
 are very nice and interesting. However, from the point of view of graphical models as well as from the viewpoint of describing social interactions, their model might be considered a bit simplistic, because every spin, i.e.~every agent in the two populations given by $S$ and $S^c$, is interacting with every other agent like in the Curie-Weiss model for ferromagnets (see \cite{Ellis-EntropyLargeDeviationsAndStatisticalMechanics}).

Indeed, this is also the set-up of the standard game theoretical models. Translated to the situation analyzed in \cite{BRS17_blockmodel} this means,
for $\beta>0$ and $\alpha \le \beta$ their model is defined by the Hamiltonian
\begin{equation}\label{standard_hamil}
\tilde H_{N,\alpha, \beta}(\sigma):= -\frac \beta {2N} \sum_{i \sim j} \sigma_i \sigma_j -\frac \alpha {2N} \sum_{i \not\sim j} \sigma_i \sigma_j, \quad \sigma \in \{-1,+1\}^N
\end{equation}
and the corresponding Gibbs measure
\begin{equation}\label{gibbs1}
\tilde \mu_{N, \alpha, \beta} (\sigma):= \frac{e^{-\tilde H_{N,\alpha, \beta}(\sigma)}}{\sum_{\sigma'}e^{-\tilde H_{N,\alpha, \beta}(\sigma')}}=:
\frac{e^{-\tilde H_{N,\alpha, \beta}(\sigma)}}{\tilde Z_{N, \alpha,\beta}}.
\end{equation}
Here we write $i \sim j$, if either $i, j \in S$ or $i,j \in S^c$ and $i \not \sim j$, otherwise. Note that this encodes interactions between every pair of spins.

One readily sees that the above model has a two-dimensional order parameter, the vector of block magnetizations, $m:=m^N:=(m^N_1, m^N_2)$, where
$$
m_1:=m^N_1:=m_1(\sigma):= \frac {2} N \sum_{i \in S} \sigma_i \qquad \mbox{and } \quad  m_2:=m^N_2:=m_2(\sigma):= \frac {2} N \sum_{i \notin S} \sigma_i.
$$

According to our interpretation of the model as a binary choice model (see the following section), we will also call $m$ the vector of {\it group decisions} or {\it group opinions}.

Indeed, one immediately sees that $m$ is an order parameter, since the Hamiltonian is handily rewritten as
\begin{equation}\label{hamil_new}
\tilde H_{N,\alpha, \beta}(\sigma)= -\frac N 8 (2 \alpha m_1 m_2 +\beta m_1^2+\beta m_2^2).
\end{equation}
In the next sections we will propose and analyze
a model that is slightly more realistic in the sense that interactions only take place between some of the spins while others do not influence each other directly (corresponding to conditional independence of the corresponding sites in a graphical model).

\section{The Model}
As mentioned above, the standard game theoretical model would be to consider groups of agents such that
all agents of all groups communicate with each other. In other words, the communication structure is described by a complete graph. This   contrasts with the  Walrasian equilibrium picture, i.e.~the traditional concept of economic equilibrium, in which
all agents communicate with the Walrasian auctioneer but not with each other directly. In this case the communication structure is star shaped with the auctioneer as the center.

However, compared to these two extreme models, it seem more plausible to assume that each agent communicates with some, though not all, other agents, thus interpolating between the two  aforementioned  models. We follow the construction in \cite{kirman83}, by assuming that  each pair of agents $(i,j)$ communicates with
probability $p=p_N$ (i.e.~there is a link between two agents with probability $p_N$), if they are in the same group and with probability  $q=q_N$, if they are in different groups. We assume that all these links exist independently from all
other links. The resulting  communication structure is then given by a corresponding inhomogeneous random graph model similar to  \cite{kirman83}.

More formally, we define indicator random variables $\varepsilon_{ij}\in \{ 0,1 \}$ for each pair of agents that are in the same group, i.e.~$i \sim j$ and, similarly, random variables $\delta_{ij}\in \{ 0,1 \}$ for each pair of agents that are in different groups, i.e.~$i \not \sim j$. The  variable $\varepsilon_{ij}$ resp.~$\delta_{ij}$ will be $1$, if agents $i$ and $j$ can communicate and $0$ otherwise.
We will assume
\begin{equation*}
\mathbb{P}(\varepsilon_{ij}=1)= 1-\mathbb{P}(\varepsilon_{ij}=0)= p_N, \quad \text{if }i\sim j
\end{equation*}
and
\begin{equation*}
\mathbb{P}(\delta_{ij}=1)= 1-\mathbb{P}(\delta_{ij}=0)= q_N, \quad   \text{if } i \not \sim j.
\end{equation*}
For technical reasons we will assume that the communication structure is asymmetric, i.e.~$\varepsilon_{ij}$
and $\varepsilon_{ji}$ are not necessarily the same (and the same for the $\delta$ variables). The case of an undirected graph, i.e.~$\varepsilon_{ij}=\varepsilon_{ji}$ and $\delta_{ij}=\delta_{ji}$,  can be dealt with using some additional arguments.
For a concise review on  random graphs in relation with economic models
see \cite{ioannides90}.

Although condition \eqref{cond_eps} and \eqref{cond_delta} are sufficient to define our model, we will specify a probability space that allows us to consider the limit $N \to \infty$.  This is in particular necessary since $\vep_{ij}=\vep_{ij}(N)$, $\delta_{ij}=\delta_{ij}(N)$, $p=p_N$ and $q=q_N$ depend on $N$.  Here, we follow the construction in \cite{BG93b}. Let $(p_N)_N$ and $(q_N)_N$ be given, non increasing sequences with $p_1=q_1=1$. For a fixed index $(i,j) \in \mathbb N \times N$ we consider independent inhomogeneous Markov chains $(\vep_{ij}(N))_{N \in \mathbb N}$ and $(\delta_{ij}(N))_{N \in \mathbb N}$ on $(\Omega_{ij}, \Sigma_{ij}, \mathbb P)$ for $\Omega_{ij}= \{0,1\}^{\mathbb N}$ with transition probabilities
\begin{align}
&\mathbb P(\vep_{ij}(N)=0 |\vep_{ij}(N-1)=0 )=0,
\nonumber \\
&\mathbb P(\vep_{ij}(N)=1 |\vep_{ij}(N-1)=0 )=0,
\nonumber \\
&\mathbb P(\vep_{ij}(N)=0 |\vep_{ij}(N-1)=1 )=1-\frac{p_N}{p_{N-1}}, \label{prob_eps_1}
\\
&\mathbb P(\vep_{ij}(N)=1 |\vep_{ij}(N-1)=1 )=\frac{p_N}{p_{N-1}}, \label{prob_eps_2}
\end{align}
resp.
\begin{align}
&\mathbb P(\delta_{ij}(N)=0 |\delta_{ij}(N-1)=0 )=0,
\nonumber\\
&\mathbb P(\delta_{ij}(N)=1 |\delta_{ij}(N-1)=0 )=0,
\nonumber
\\
&\mathbb P(\delta_{ij}(N)=0 |\delta_{ij}(N-1)=1 )=1-\frac{q_N}{q_{N-1}},\label{prob_delta_1}
\\
&\mathbb P(\delta_{ij}(N)=1 |\delta_{ij}(N-1)=1 )=\frac{q_N}{q_{N-1}}.\label{prob_delta_2}
\end{align}
The probabilities in \eqref{prob_eps_1} and \eqref{prob_eps_2} resp.~\eqref{prob_delta_1} and \eqref{prob_delta_2} are chosen such that they imply \eqref{cond_eps} resp.~\eqref{cond_delta}.
We consider the product  probability space  $(\Omega, \Sigma)$ with $\Omega := \prod_{(i,j) \in \mathbb N \times \mathbb N} \Omega_{ij}, \Sigma := \prod_{(i,j) \in \mathbb N \times \mathbb N} \Sigma_{ij}$ and we set $\mathcal E:= \Omega \times \Omega$.  Then we consider $(\vep, \delta)=(\vep(N), \delta(N))$ as a family of random variables on $\{(\vep, \delta) \in \mathcal E: \vep_{ij} \in \{0,1\}^{\mathbb N} \text{ for all } i,j \in \{1, \ldots, N\}, i\sim j \text{and }  \delta_{ij} \in \{0,1\}^{\mathbb N} \text{ for all } i,j \in \{1, \ldots, N\}, i\nsim j  \}$.

To describe the decision making process, each agent has the choice from a discrete set of
alternatives, which in our situation will be the set $\{-1,+1\}$ as in \cite{foellmer74,brockdurlauf,schelling71}.
E.g., agents may have to choose between two different product brands or between two different political parties as in \cite{kirschpolitik}.

Now, we assume that for an agent $i$ the choice $\sigma_i$ maximizes a certain utility function $U_i:\{-1,+1\}\to \mathbb{R}$.
To interpolate between pure individual choice and pure peer pressure decisions,
we suppose that the utility
function $U_i$ has two components: an individual
part $I_i(\sigma_i)$, which only depends on $\sigma_i$, and a common piece $C_i$, which
also depends on the choices of all other individuals $\sigma_j , j\not=
i$, that communicate with agent $i$. We thus write
\begin{equation}\label{utility}
U_i(\sigma_i):= I_i(\sigma_i) + C_i(\sigma_i,\{\sigma_j , j\not= i \}).
\end{equation}
The choice of $\sigma_i\in \{-1,1\}$ implies that we can take $I_i$ to be a linear function, which we will write as
\begin{equation*}
I_i(\sigma_i)= u_i\sigma_i + h\sigma_i.
\end{equation*}
Here, $h$ is the same for all agents and  expresses the apriori tendency to vote ``yes" or ``no''.
We will exclusively consider the case $h\equiv 0$ for the same reasons as we will stick to equal block sizes and two groups $S$ and $S^c$, only: Otherwise even for the full model precise conditions for the existence of a phase transition are unknown.

To model that individual tastes are heterogeneous we take our utility functions random as in
\cite{hildenbrand}. More precisely, we assume that $(u_i)_{i=1 \ldots N}$ are
i.i.d.~random variables with common distribution function $F$. As often we will assume that $F$ has a {\it logit} distribution
(see e.g.~\cite{anderson}):
\begin{equation}\label{logit}
F(x)= \mathbb{P}(u_i\leq x):= \frac{1}{1+
\exp (-\beta{}x)}. 
\end{equation}
Here $\beta >0$ describes the homogeneity of the preferences of the agents:
Large values of $\beta$ describe a group in which the group members share similar
tastes.

In order to describe the second component of our utility function $U_i$ in \eqref{utility} we need to define a neighborhood $\Ns_i$ for each agent $i$.
These are the individuals that directly communicate with agent $i$ and hence have a direct influence on his or her utility function. Mathematically speaking
\begin{equation*}
\Ns_i:=\{ j \ | \quad \varepsilon_{ij}= 1 \mbox{  or   } \delta_{ij}=1\}.
\end{equation*}
We can partition $\Ns_i$ into those $j$ that belong to the same group as $i$
\begin{equation*}
\Ns^{\sim}_i:=\{ j \ | \quad \varepsilon_{ij}= 1 \}
\end{equation*}
and those that belong to a different group
\begin{equation*}
\Ns^{\not\sim}_i:=\{ j \ | \quad \delta_{ij}= 1 \}.
\end{equation*}
Note that since our communication structure is random, so are the sets $\Ns_i$, and in particular, they will be different for different agents $i$.
For the second part of our utility function $U_i$, which we called $C_i$, we will assume
an additive structure as in \cite{foellmer74,brockdurlauf}:
\begin{eqnarray*} 
C_i(\sigma_i,\{\sigma_j , j\in \Ns_i \}) &=& C^{\sim}_i(\sigma_i,\{\sigma_j , j\sim i \})+C^{\not\sim}_i(\sigma_i,\{\sigma_j , j\not\sim i \})
\nonumber \\
& =  &\frac{1}{pN} \sum_{j\in\Ns^{\sim}_i} \sigma_j \sigma_i + \frac{\alpha}{\beta pN} \sum_{j\in\Ns^{\not \sim}_i} \sigma_j \sigma_i.
\end{eqnarray*}
Note that we normalized the interaction term stemming from $C^{\sim}_i$ by the expected size of $\Ns^{\sim}_i$, which is $pN$ (upto a factor $\frac{1}{2}$). The second summand is also normalized by the same factor $pN$, hence reflecting that typically members of different groups have less influence on agent $i$'s choice than members of the same group (if $q_N < p_N$). This choice in particular ensures that in the limit $N \to \infty$ the social utility $C_i$
does not systematically dominate the individual utility $I_i$. Moreover, $\alpha$ is a real parameter with
$|\alpha | \le \beta$. We added this parameter to also allow for contrasting votes in the two groups. We divided the second summand by $\beta$ to obtain a nice form of the invariant measure, see the following section. Of course, one could imagine other parameterizations, e.g.~dividing the second summand by $qN$ instead of $pN$ or not dividing by $\beta$. This would lead to other critical values for the parameters, but, of course, it would not change the overall picture.

In approaching the equilibrium picture of the above model, we will now assume that each agent $i$
makes his or her choice $\sigma_i$ by maximizing his or her
utility function $U_i$ as given in (\ref{utility}). Observe that the second summand in $U_i$, i.e.~$C_i$, depends on the choices of all other agents in her neighborhood, $\sigma_j, j \not= i, j \in \Ns_i$.
To maximize $U_i$ we will therefore assume that all $\sigma_j, j \not= i, j \in \Ns_i$ are fixed and maximize $U_i$ conditionally on
$\sigma_j, j \not= i, j \in \Ns_i$.
Then, given $\sigma_j, j \not= i, j \in \Ns_i$, agent $i$ will decide for
$\sigma_i=+1$ if
\begin{equation*} I_i(+1) + C_i(+1,\{\sigma_j , j\in\Ns_i
\})
> I_i(-1) + C_i(-1,\{\sigma_j , j\in\Ns_i
\}).
\end{equation*}
This is obviously the case, if and only if
\begin{equation*}
I_i(+1) - I_i(-1) > C_i(-1,\{\sigma_j , j\in\Ns_i \}) -
C_i(+1,\{\sigma_j , j\in\Ns_i \}), \nonumber
\end{equation*}
which in turn is fulfilled, if and only if
\begin{equation*}
u_i >  -\frac{1}{pN}\sum_{j\in\Ns^{\sim}_i} \sigma_j -\frac{\alpha}{\beta pN}\sum_{j\in\Ns^{\not\sim}_i} \sigma_j,
\end{equation*}
where we set $h \equiv 0$.

The conditional probability of agent $i$ choosing $\sigma_i=1$, given all the other decisions $\sigma_j, j \not= i, j \in \Ns_i$, can
therefore be computed as:
\ba\label{conditional}
\mathbb{P}(\sigma_i=+1|(\sigma_j)_{j\in\Ns_i}) &=& 1-
\mathbb{P}(u_i\leq -\frac{1}{pN}\sum_{j\in\Ns^{\sim}_i} \sigma_j -\frac{\alpha}{\beta pN}\sum_{j\in\Ns^{\not\sim}_i} \sigma_j)\nonumber \\
&=& 1- F(-\frac{1}{pN}\sum_{j\in\Ns^{\sim}_i} \sigma_j -\frac{\alpha}{\beta pN}\sum_{j\in\Ns^{\not\sim}_i} \sigma_j).
\ea
If $F$ is the logit distribution given by (\ref{logit}), our condition (\ref{conditional})
turns into:
 \ba\label{logit.cond}
&&\mathbb{P}(\sigma_i=+1|\ (\sigma_j)_{j\in\Ns_i})  =
 \frac{\exp{\beta C_i(+1,\sigma_j)}}{
\exp({\beta C_i(+1,\sigma_j)}) + \exp({\beta
C_i(-1,\sigma_j)}) } \\
&=& \frac{ \exp[\frac{\beta}{2}(\frac{1}{pN}\sum_{j\in\Ns^{\sim}_i} \sigma_j +\frac{\alpha}{\beta pN}\sum_{j\in\Ns^{\not\sim}_i} \sigma_j)]}{\exp[\frac{\beta}{2}(\frac{1}{pN}\sum_{j\in\Ns^{\sim}_i} \sigma_j +\frac{\alpha}{\beta pN}\sum_{j\in\Ns^{\not\sim}_i} \sigma_j)] +
\exp[-\frac{\beta}{2}(\frac{1}{pN}\sum_{j\in\Ns^{\sim}_i} \sigma_j +\frac{\alpha}{\beta pN}\sum_{j\in\Ns^{\not\sim}_i} \sigma_j)]}\nonumber\\
&=& \frac{ \exp[\frac{\beta}{2pN}\sum_{j\in\Ns^{\sim}_i} \sigma_j +\frac{\alpha}{2 pN}\sum_{j\in\Ns^{\not\sim}_i} \sigma_j]}{\exp[\frac{\beta}{2pN}\sum_{j\in\Ns^{\sim}_i} \sigma_j +\frac{\alpha}{2 pN}\sum_{j\in\Ns^{\not\sim}_i} \sigma_j] +
\exp[-\frac{\beta}{2pN}\sum_{j\in\Ns^{\sim}_i} \sigma_j -\frac{\alpha}{2pN}\sum_{j\in\Ns^{\not\sim}_i} \sigma_j]},\nonumber
\ea
which is the form of a Glauber dynamics in statistical physics.

Note that still the decisions of the agents are taken randomly, which can be interpreted as a heterogeneity in the decision taking within our population of the two groups. However, it can also be understood as a randomness that is inherent in the agent's choice, because they are not acting completely rational.

Let us briefly discuss some particular choices of the parameters in this model:
When $C_i$ is solely a function of $\sigma_i$ the model reduces to the standard
logit model. This is well known from the literature on discrete choice models, see e.g.~\cite{anderson}. In contrast, if $C_i$ indeed depends on $\sigma_j, j \in \Ns_i$ the above equations
(\ref{conditional}) and (\ref{logit.cond}) symbolize the
influence of the social environment on an agent's decision via the conditional
distribution $\mathbb{P}(\sigma_i|\sigma_j, j~\in\Ns_i)$.

As for the influence parameter $\beta$, for a very large value of $\beta$ (i.e.~the case $\beta\to\infty$)
the model represents the classical utility maximizer. However, in this case the utility is solely determined by the social utility function
and completely ignores individual tastes. Therefore the model could be expected to show very similar decisions for the
various individuals. This is to be contrasted with the case $\beta=0$. Here agents will choose any of the two possible decisions with equal probabilities. This is reflected in
a very heterogeneous picture of the decisions of the agents.

In order to interpolate  between
these two extremes, we study positive, but finite values of $\beta$ and we are particularly interested in how the behavior of the agents
changes for different values of $\beta\in (0,\infty)$.

Moreover, as we will see in our analysis the product of $\alpha$ and the limit of the ratio $q_N/p_N$ determines the mutual influence of the two groups.

\section{The invariant measure}
So far, we described how a fixed agent $i$ takes his or her decision.
A natural follow-up question is, when the system is in equilibrium. Because of the interdependence of the individual decisions this question is non-trivial.

Many approaches deal with a notion of equilibrium that is defined by self-consistency of the
actions or beliefs \cite{brockdurlauf,GS00,horstscheinkman}.
This notion of a static equilibrium defines
a configuration of decisions
$(\sigma_i^*)_{i=1 \ldots N}$ to be in a (static) equilibrium if, for each $i$, $\sigma_i^*$ is the best response
to the decisions of the other agents $\sigma_j^*$. That is
\begin{equation*}  \mbox{For all }
i=1 \ldots N,\quad \sigma_i^* \in \mathop{\arg\max}_{\sigma_i\in {-1,+1}}
U_i(\sigma_i,\{\sigma^*_j,j\neq i\}).
\end{equation*}

Note, however, that such a definition
of a ``self-consistent equilibrium'' is also static in the sense that it gives no clue as to how such an equilibrium point can be reached. Indeed, the mere existence of an equilibrium point does not imply that
it can be reached from some given starting configuration (which is not in equilibrium).

We will therefore take a dynamic approach and endow all agents with a Poisson clock. When agent $i$'s alarm goes off at time $t$, she will take a new decision and update her opinion $\sigma_i$ according to \eqref{logit.cond}. Here the $\sigma_j, j \in \Ns_i$ are the decisions of the other agents at time $t$. In this way we can construct a continuous time Markov chain. This chain is ergodic and by detailed balance it is immediately checked that its invariant measure is given by the measure $\mu$ on the state space $\Sigma_N:=\{-1,+1\}^N$:
\begin{equation}\label{gibbs2}
\mu(\sigma):=\mu_{N,\alpha,\beta,\vep,\delta,S}(\sigma):= \frac{e^{- H_{N,\alpha, \beta, \vep, \delta,S}(\sigma)}}{\sum_{\sigma'}e^{-  H_{N,\alpha, \beta, \vep, \delta,S}(\sigma')}}=:
\frac{e^{- H_{N,\alpha, \beta, \vep, \delta,S}(\sigma)}}{Z_{N,\alpha,\beta, \vep,\delta,S}}.
\end{equation}
Here the Hamiltonian of the Gibbs measure is defined by the following function on
$\Sigma_N:=\{-1,+1\}^N$:
\begin{equation*}
H(\sigma):=H_{N,\alpha, \beta,\vep, \delta,S}(\sigma):= -\frac \beta {2Np} \sum_{i \sim j} \vep_{ij} \sigma_i \sigma_j -\frac \alpha {2Np} \sum_{i \not\sim j} \delta_{ij} \sigma_i \sigma_j, \quad \sigma \in \Sigma_N.
\end{equation*}
Moreover, let us quickly repeat our conditions on the parameters: We chose $\beta>0$, the set $S \subset \{1, \ldots, N\}$ has cardinality $\frac N2$, and $\vep:=\vep_N:=(\vep_{ij})_{{i,j} \subset S \text{ or } {i,j} \subset S^c}$ and $\delta:=\delta_N:=(\delta_{ij})_{(i,j) \in S \times S^c \text{ or } (i,j) \in S^c \times S }$ are independent Bernoulli random variables.
Moreover the $\vep$-variables and the $\delta$-variables are independent from each other. Their distribution is given by
\begin{equation}\label{cond_eps}
\P(\vep_{ij} =1)=1-\P(\vep_{ij} =0)=p=p_N
\end{equation}
as well as
\begin{equation}\label{cond_delta}
\P(\delta_{ij} =1)=1-\P(\delta_{ij} =0)=q=q_N
\end{equation}
and we will assume that $p \ge q$ and that $|\alpha|\frac{q}{p}\le \beta$, to model that the influence within a group is stronger than across two groups. Recall the construction of the underlying probability space for the $\vep$- and $\delta$-variables in the previous section. 

Moreover, the model described above implies that we will consider the so-called quenched situation, where the realisations of $\vep$ and $\delta$ are tossed in advance and then fixed for the rest of the considerations. Note that this constitutes a Ising block model on a {\it directed } random graph. As mentioned above this is basically a way to make the computations slightly more convenient. The undirected graph case can also be treated.
Finally we will take the liberty and omit indices, whenever we think that this is reasonable.

The goal will be to analyze this model (which is in the spirit of \cite{LV10} e.g.) for $p_N$ and $q_N$ large enough, to describe its statistical mechanics.

The first theorem we will prove is the following:
\begin{theorem}\label{theo_stat_mech}
Assume that
\begin{equation*}{p_N N} \to \infty \qquad \mbox{and that }\quad \frac{q_N}{p_N} \to a \in [0,1] \quad \mbox{as $N \to \infty$.}
\end{equation*}
Then, there  are subsets
\begin{equation*}
\mathcal{E}_N^*\subseteq \mathcal{E}_N := \left\{(\vep, \delta):\vep \in \{0,1\}^{(S \times S) \cup (S^c \times S^c)}, \delta \in \{0,1\}^{S \times S^c}\right\}
\end{equation*}
with
\begin{equation*}
\P(\mathcal{E}^*)=1\quad \text{for } \quad \mathcal{E}^*:= \liminf_{N \to \infty} \mathcal{E}^*_N =\bigcup_{M=1}^\infty  \bigcap_{N=M}^\infty \mathcal{E}^*_N,
\end{equation*}
such that for all sequences $(\vep,\delta)\in \mathcal{E}^*$ the vector of group opinions $m=(m_1, m_2)$ satisfies:
\begin{itemize}
\item If $\beta+ |\alpha a| \le 2$ the distribution of $m$ under the Gibbs measure $\mu_{N}$, i.e.~$\mu_{N} \circ m^{-1}$ converges weakly to the Dirac measure in $(0,0)$.
\item If $\beta >2$ and $\alpha a=0$, then $\mu_{N,\beta} \circ m^{-1}$ converges weakly to the mixture of four Dirac measures
$\frac 1 4 \sum_{v_1, v_2 \in \{\pm\}} \delta_{(v_1 z^*(\frac \beta 2), v_2 z^*(\frac \beta 2))}$.
Here $z^*(b)$ denotes the largest solution of the Curie-Weiss equation
\begin{equation*}
z= \tanh(b z)
\end{equation*}
(which is positive, if $b>1$).
\item If $\beta+\alpha a > 2$, $\alpha a>0$ and $a \neq 0$, the distribution of $m$ under the Gibbs measure $\mu_{N,\beta}$ converges weakly to the following mixture of two Dirac measures
$$\frac 1 2 \sum_{v \in \{\pm\}} \delta_{(v z^*(\frac {\beta+\alpha a} 2), v z^*(\frac{\beta+\alpha a} 2))}.$$
\item If $\beta+ \alpha a > 2$, $\alpha a <0$ and $a \neq 0$, the distribution of $m$ under the Gibbs measure $\mu_{N,\beta}$ converges weakly to the following mixture of two Dirac measures
$$\frac 1 2 \sum_{v \in \{\pm\}} \delta_{(v z^*(\frac {\beta+\alpha a} 2), -v z^*(\frac{\beta+\alpha a} 2))}.$$
\end{itemize}
\end{theorem}

Theorem \ref{theo_stat_mech} tells us that there is a phase transition for the vector of group opinions. If both, $\beta$ and $|a \alpha|$ are small enough, i.e. if $\beta+|\alpha a| \le 2$, then on a set with huge probability for both groups the average group opinion will behave as if decisions were taken independently with probability $1/2$ for $+1$ and $-1$ (however, the fluctuations are different). If $\beta >2 $ and $\alpha a=0$, there are four different limit points for the group opinions. This is reasonable because each group behaves similar to a Curie-Weiss model at low temperature (where there are two limit points of the magnetizations) and $\alpha a=0$ indicates that the group opinions are asymptotically independent. If $\beta >2 $ and $|\alpha a| >0 $, there
are two possible non-zero limit points for the vector of group opinions and the decisions of the two groups are positively correlated, if $\alpha a >0 $ and negatively correlated, when $\alpha a < 0 $.

We will prove Theorem \ref{theo_stat_mech} in the next section.
We will also mention a consequence of our proof that allows to derive the free energy of our model.

\section{Proof of Theorem \ref{theo_stat_mech}}\label{proof}
In this section we will prove Theorem \ref{theo_stat_mech}. Its proof basically relies on the law of large numbers for the coupling variables $\vep$ and $\delta$. More precisely, we consider subsets of $\mathcal{E}_N$,  for which there are large subsets of the vertices in which there are much more or much less edges than expected. These sets have exponentially small probabilities. A similar argument was made in \cite{BG93b}.

For a fixed configuration $\sigma \in \{\pm1\}^N$ let us introduce the sets of sites aligned and unaligned spins, both within the same  block (indicated by the subscript `b' in the notation below) and in different blocks (indicated by the subscript `nb' for `not the same block' below). These are denoted by
\begin{equation*}
L^{\pm}_b(\sigma):= \{i\sim j: \sigma_i \sigma_j=\pm 1\}
\end{equation*}
as well as
\begin{equation*}
L^{\pm}_{nb} (\sigma):= \{i\not\sim j: \sigma_i \sigma_j=\pm 1\}.
\end{equation*}
Then we are able to express $\tilde H_{N,\alpha, \beta}(\sigma)$ in terms of $L^+_b$ and
$L^+_{nb}$, only. Indeed: Note that
\begin{eqnarray*}
m_1^2+m_2^2 &=& \frac 4 {N^2} \left(\sum_{i,j \in S} \sigma_i \sigma_j +\sum_{i,j \notin S} \sigma_i \sigma_j \right) =  \frac 4 {N^2}(|L^+_b|-|L^-_b|)\\
&=& \frac 4 {N^2}\left(|L^+_b|-\left(2 \frac {N^2}{ 4}-|L^+_b|\right)\right) = \frac 8 {N^2} |L^+_b|-2.
\end{eqnarray*}
Hence
\begin{equation}\label{L1}
|L^+_b|=\frac{N^2}8 (m_1^2+m_2^2+2).
\end{equation}

Similarly,
\begin{eqnarray*}
m_1\,m_2 &=& \frac 4 {N^2} \left(\sum_{i\in S}\sum_{j \in S^c} \sigma_i \sigma_j \right) =  \frac 2 {N^2}(|L^+_{nb}|-|L^-_{nb}|)\\
&=& \frac 2 {N^2}\left(|L^+_{nb}|-\frac {N^2}{ 2}+|L^+_{nb}|\right) = \frac 4 {N^2} |L^+_{nb}|-1.
\end{eqnarray*}
This gives
\begin{equation*}
|L^+_{nb}|=\frac{N^2}4 (m_1m_2+1).
\end{equation*}
Analogously,
\begin{eqnarray*}
m_1\,m_2 &=&  \frac 2 {N^2}(|L^+_{nb}|-|L^-_{nb}|)\\
&=& \frac 2 {N^2}\left(\frac {N^2}{ 2}-2|L^+_{nb}|\right) =1- \frac 4 {N^2} |L^-_{nb}|.
\end{eqnarray*}
Thus
\begin{equation}\label{L3}
|L^-_{nb}|=\frac{N^2}4 (1-m_1m_2).
\end{equation}

On the other hand, making use of
\begin{eqnarray*}
\sum_{i \sim j} \vep_{ij} \sigma_i \sigma_j&= & \sum_{i \sim j} \vep_{ij} \sigma_i \sigma_j \mathbbm{1}_{i,j \in L^+_b}+\sum_{i \sim j} \vep_{ij} \sigma_i \sigma_j \mathbbm{1}_{i,j \in L^-_b}\\
&=& \sum_{i \sim j} \vep_{ij} \mathbbm{1}_{i,j \in L^+_b}-\sum_{i \sim j} \vep_{ij} (1-\mathbbm{1}_{i,j \in L^+_b})\\
&=&  \sum_{i \sim j} \vep_{ij} (2\mathbbm{1}_{i,j \in L^+_b}-1)\\
&=&  \sum_{i \sim j} \vep_{ij} 2\mathbbm{1}_{i,j \in L^+_b}- \sum_{i \sim j} \vep_{ij}
\end{eqnarray*}
and a similar observation for $\sum_{i \not\sim j} \delta_{ij} \sigma_i \sigma_j$
we can rewrite the Hamiltonian in terms of $L^+_b$ and $L^+_{nb}$:
\begin{eqnarray}\label{rewrite}
H(\sigma)&=& -\frac \beta {2Np} \sum_{i \sim j} \vep_{ij} \sigma_i \sigma_j -\frac \alpha {2Np} \sum_{i \not\sim j} \delta_{ij} \sigma_i \sigma_j  \\
&=& -\frac 1 {2Np} \left(\sum_{i \sim j} 2\beta\vep_{ij} \mathbbm{1}_{i,j \in L^+_b}- \sum_{i \sim j} \beta \vep_{ij}+
\sum_{i \not\sim j} 2\alpha \delta_{ij} \mathbbm{1}_{i,j \in L^+_{nb}}- \sum_{i\not \sim j} \alpha \delta_{ij}\right).\nonumber
\end{eqnarray}
In the same way, we can find an expression for $H(\sigma)$ that uses $L^-_{nb}$ instead of $L^+_{nb}$:
\begin{eqnarray}\label{rewrite2}
H(\sigma)
= -\frac 1 {2Np} \left(\sum_{i \sim j} 2\beta\vep_{ij} \mathbbm{1}_{i,j \in L^+_b}- \sum_{i \sim j} \beta \vep_{ij}+\sum_{i \not\sim j} \alpha \delta_{ij}-
\sum_{i \not\sim j} 2\alpha \delta_{ij} \mathbbm{1}_{i,j \in L^-_{nb}}\right).
\end{eqnarray}

We will thus show that for a subset of $\mathcal{E}_N$ with huge probability the sizes of the sets $L^+_b$ and $L^+_{nb}$ or $L^+_b$ and $L^-_{nb}$ are of the order we would expect them to be.
The point, why we just need $L^+_b$, but both, $L^+_{nb}$ and $L^-_{nb}$, is that $L^+_b$ is automatically of order $N^2$, while $L^+_{nb}$ or $L^-_{nb}$ may be of much smaller order, but they cannot both be small.

More precisely, for two sequences $\gamma_N >0$ and $\kappa_N >0$ that we will specify in the following proposition we define
\begin{eqnarray*}
\mathcal{E}^*_N &:=& \mathcal{E}^*_{b,N} \cap \mathcal{E}^*_{nb^+,N}\cap \mathcal{E}^*_{nb^-,N}\\
&=& (\mathcal{E}^*_{b,N} \cap \mathcal{E}^*_{nb^+,N})\cap (\mathcal{E}^*_{b,N} \cap \mathcal{E}^*_{nb^-,N})
\end{eqnarray*}
Here
\begin{eqnarray*}
\mathcal{E}^*_{b,N}:=
\left\{(\vep, \delta)\in \mathcal{E}: \forall \sigma \in\{\pm 1\}^N:
\left|\sum_{i \sim j} \vep_{ij} \mathbbm{1}_{i,j \in L^+_b}-p_N |L^+_b| \, \right| \le \gamma_N p_N |L^+_b| \right\},
\end{eqnarray*}
\begin{eqnarray*}
&&\hskip-.5cm \mathcal{E}^*_{nb^+,N}:=\\
&&\hskip-.5cm \left\{ (\vep, \delta)\in \mathcal{E}: \forall \sigma \in\{\pm 1\}^N \mbox{with }m_1(\sigma)m_2(\sigma) \ge0: \left|\sum_{i \sim j} \delta_{ij} \mathbbm{1}_{i,j \in L^+_{nb}}-q_N |L^+_{nb}| \, \right| \le \kappa_N q_N |L^+_{nb}| \right\},
\end{eqnarray*}
as well as
\begin{eqnarray*}
&&\hskip-.5cm \mathcal{E}^*_{nb^-,N}:=\\
&&\hskip-.5cm \left\{ (\vep, \delta)\in \mathcal{E}: \forall \sigma \in\{\pm 1\}^N \mbox{with }m_1(\sigma)m_2(\sigma) < 0: \left|\sum_{i \sim j} \delta_{ij} \mathbbm{1}_{i,j \in L^-_{nb}}-q_N |L^-_{nb}| \, \right| \le \kappa_N q_N |L^-_{nb}| \right\}.
\end{eqnarray*}
Finally set
\begin{equation*}
\mathcal{E}^{*,+}_N:=\mathcal{E}^*_{b,N} \cap \mathcal{E}^*_{nb^+,N} \qquad \mbox{and  }\quad
\mathcal{E}^{*,-}_N:=\mathcal{E}^*_{b,N} \cap \mathcal{E}^*_{nb^-,N}.
\end{equation*}
The desired set $\mathcal{E}^*$ is now given by the $\liminf$ of the sets $\mathcal{E}^*_N$:
$$
\mathcal{E}^*:= \liminf_{N \to \infty}\mathcal{E}^*_N = \bigcup_{M=1}^\infty  \bigcap_{N=M}^\infty \mathcal{E}^*_N.
$$
We now show that $\mathcal{E}^*$ has full probability:
\begin{proposition}\label{prop_full_measure}
If $p_N$ and $q_N$ satisfy the assumptions of Theorem \ref{theo_stat_mech} and $\gamma_N \ge \frac c{\sqrt{p_N N}}$ and $\kappa_N \ge \frac d {\sqrt{q_N N}}$ for some $c,d>0$ to be chosen later and $\gamma_N$ as well as $\kappa_N$ tend to 0, then
$$
\P(\mathcal{E}^*)=1.
$$
\end{proposition}
Not unexpectedly Proposition \ref{prop_full_measure} will follow from an estimate for the probabilities of $\mathcal{E}_N^*$ and the Borel-Cantelli-Lemma. The needed estimate for the proof of Proposition \ref{prop_full_measure} is provided in the following lemma.
\begin{lemma}
Under the assumptions of Proposition \ref{prop_full_measure} we have that there exist two constants $C_1, C_2>0$ such that for all $N$ large enough
$$
\P((\mathcal{E}_N^*)^c)\le 4 C_1  N^2 \exp(- C_2 N).
$$
\end{lemma}
\begin{proof}
Assume that $(\vep, \delta) \notin \mathcal{E}_N^*$.
Then there exists a $\sigma \in \{\pm 1\}^N$ such  either
\begin{equation*}  
\left|\sum_{i \sim j} \vep_{ij} \mathbbm{1}_{i,j \in L^+_b(\sigma)}-p_N |L^+_b(\sigma)| \, \right| > \gamma_N p_N |L^+_b(\sigma)|
\end{equation*}
or
$$
\left|\sum_{i \sim j} \delta_{ij} \mathbbm{1}_{i,j \in L^+_{nb}(\sigma)}-q_N |L^+_{nb}(\sigma)| \, \right| > \kappa_N q_N |L^+_{nb}(\sigma)|, \quad  m_1(\sigma) m_2(\sigma) \ge 0.
$$
or
$$
\left|\sum_{i \sim j} \delta_{ij} \mathbbm{1}_{i,j \in L^-_{nb}(\sigma)}-q_N |L^-_{nb}(\sigma)| \, \right| > \kappa_N q_N |L^-_{nb}(\sigma)|, \quad m_1(\sigma) m_2(\sigma) < 0.
$$

Hence by a union bound
\begin{eqnarray}\label{sum1}
\P((\mathcal{E}_N^*)^c)&\le & \sum_{\sigma: m_1 m_2 \ge 0} \P(A_N(\sigma))+\P(B_N(\sigma))+\P(C_N(\sigma))+\P(D_N(\sigma))\\
&+& \sum_{\sigma: m_1 m_2 < 0} \P(A_N(\sigma))+\P(B_N(\sigma))+\P(C'_N(\sigma))+\P(D'_N(\sigma)).\nonumber
\end{eqnarray}
Here
\begin{eqnarray*}
A_N(\sigma) &:=& \{(\vep,\delta): \sum_{i \sim j} \vep_{ij} \mathbbm{1}_{i,j \in L^+_b} \ge p_N (1+ \gamma_N) |L^+_b|\},\\
B_N(\sigma) &:=& \{(\vep,\delta): \sum_{i \sim j} \vep_{ij} \mathbbm{1}_{i,j \in L^+_b} \le p_N (1- \gamma_N)|L^+_b|\},\\
C_N(\sigma) &:=& \{(\vep,\delta): \sum_{i \sim j} \delta_{ij} \mathbbm{1}_{i,j \in L^+_{nb}} \ge q_N (1+ \kappa_N) |L^+_{nb}|\},\\
D_N(\sigma) &:=& \{(\vep,\delta): \sum_{i \sim j} \delta_{ij} \mathbbm{1}_{i,j \in L^+_{nb}} \le q_N (1- \kappa_N) |L^+_{nb}|\}
\\
C'_N(\sigma) &:=& \{(\vep,\delta): \sum_{i \sim j} \delta_{ij} \mathbbm{1}_{i,j \in L^-_{nb}} \ge q_N (1+ \kappa_N) |L^-_{nb}|\}\quad \mbox{and}\\
D'_N(\sigma) &:=& \{(\vep,\delta): \sum_{i \sim j} \delta_{ij} \mathbbm{1}_{i,j \in L^-_{nb}} \le q_N (1- \kappa_N) |L^-_{nb}|\}.
\end{eqnarray*}
Defining the relative entropy
\begin{align*}
I_p(x&)= x \log \frac x p + (1-x) \log \frac{(1-x)}{(1-p)}  \quad \text{and}
\\
I_q(x)&= x \log \frac x q + (1-x) \log \frac{(1-x)}{(1-q)},
\end{align*}
we obtain by an exponential Chebyshev inequality
\begin{eqnarray*}
\P(A_N) & \le & \exp \left(- |L^+_b(\sigma)| I_p(p_N (1+\gamma_N))\right),\\
\P(B_N) & \le & \exp \left(- |L^+_b(\sigma)| I_p(p_N (1-\gamma_N))\right),\\
\P(C_N) & \le & \exp \left(- |L^+_{nb}(\sigma)| I_q(q_N (1+\kappa_N))\right),\\
\P(D_N) & \le & \exp \left(- |L^+_{nb}(\sigma)| I_q(q_N (1-\kappa_N))\right),\\
\P(C'_N) & \le & \exp \left(- |L^-_{nb}(\sigma)| I_q(q_N (1+\kappa_N))\right)\qquad \mbox{and }\\
\P(D'_N) & \le & \exp \left(- |L^-_{nb}(\sigma)| I_q(q_N (1-\kappa_N))\right).\\
\end{eqnarray*}
In order   to  keep the notation simple, here  we write $I_p$ instead of $I_{p_N}$ in the last formula and $I_q$ instead of $I_{q_N}$.  Note  that $I_p(x)$ is always positive.
Moreover, the quantities $|L^+_b(\sigma)|$, $|L^+_{nb}(\sigma)|$, and $|L^-_{nb}(\sigma)|$ can be expressed in terms of the vector of magnetizations $m$ as observed in \eqref{L1} to \eqref{L3}:
$$
|L^+_b|=\frac{N^2}8 (m_1^2+m_2^2+2)
$$
as well as
$$
|L^+_{nb}|=\frac{N^2}4 (m_1m_2+1)
$$
and
$$
|L^-_{nb}|=\frac{N^2}4 (1-m_1m_2).
$$

We will start by estimating the contributions from the first line in \eqref{sum1}.
Thus decomposing this first sum in \eqref{sum1} according to which vector of magnetizations $m(\sigma)$ we obtain from $\sigma$ and applying the exponential bounds derived above, we arrive at:
\begin{eqnarray*}
&&\P((\mathcal{E}_N^{*,+})^c)\le \sum_{\sigma \atop m_1(\sigma) m_2(\sigma) \ge 0} \P(A_N(\sigma))+\P(B_N(\sigma))+\P(C_N(\sigma)) +\P(D_N(\sigma))\\
&=&\hskip-.5cm  \sum_{m_1, m_2 \atop m_1m_2 \geq 0}\, \sum_{\sigma: m_1(\sigma)=m_1, m_2(\sigma)=m_2} \P(A_N(\sigma))+\P(B_N(\sigma))+\P(C_N(\sigma)) + \P(D_N(\sigma))
\\
&\leq & \hskip-.5cm \sum_{m_1, m_2 \atop m_1 m_2 \ge 0} \binom{\frac N2}{\frac N4 (1+m_1)}\binom{\frac N2}{\frac N4 (1+m_2)}\times\\
&&\hskip-.5cm \times \left\{ \exp \left(-\frac{N^2}8 (m_1^2+m_2^2+2)  I_p(p_N (1+\gamma_N))\right)+
\exp \left(-\frac{N^2}8 (m_1^2+m_2^2+2)  I_p(p_N (1-\gamma_N))\right)
\right.\\
&& \, + \left. \exp \left(-\frac{N^2}4 (m_1m_2+1)  I_q(q_N (1+\kappa_N))\right)+
\exp \left(-\frac{N^2}4 (m_1m_2+1)I_q(q_N (1-\kappa_N))\right)
\right\}.
\end{eqnarray*}
Here and in the sequel, the sums over $m_1$ and $m_2$ are over such values for the $m_i$ that are admissible magnetizations for the given $N$. More precisely, admissible magnetizations means that for $\frac{N}{2}$ even  we consider $\frac{N}{2}m_1, \frac{N}{2}m_2 \in \{0, \pm 2, \pm4, \ldots, \pm \frac{N}{2} \}$ and for $\frac{N}{2}$ odd we consider $\frac{N}{2}m_1, \frac{N}{2}m_2 \in \{\pm1 , \pm 3,  \ldots, \pm \frac{N}{2} \}$.

By Stirling's formula $1 \leq \frac{n!}{\sqrt{2 \pi n}(\frac{n}{e})^n}\leq 2$ 
we have  for $\gamma \in (0,1)$ such that $\gamma M$ is an integer 
\begin{align}
&\binom{M}{\gamma M}
\le C \frac{\sqrt{M} (M/e)^{M}}{\sqrt{\gamma M} \gamma^{\gamma M} M^{\gamma M} \sqrt{(1-\gamma)M} (1-\gamma)^{(1-\gamma)M} M^{(1-\gamma)M} (1/e)^M}
\\& =   C  \frac{1}{\sqrt{\gamma (1-\gamma) M}} \exp(-M (\gamma \log \gamma + (1-\gamma)\log(1-\gamma)))
\end{align}
for some constant $C>0$.
Recall that we want to apply this estimate to  $M=\frac{N}{2}$ and $\gamma =\frac{1+ m_1}{2}$ resp. $\gamma =\frac{1+ m_2}{2}$, i.e.~$\gamma$ takes values in the set $\Gamma_1:=\{\frac{1}{2}, \frac{1}{2} \pm \frac{2}{N}, \frac{1}{2} \pm \frac{4}{N}, \ldots, \frac{1}{2}\pm (\frac{1}{2}-\frac{2}{N}) \}$ for $\frac{N}{2}$ even resp.~in the set $\Gamma_2:=\{\frac{1}{2} \pm \frac{1}{N}, \frac{1}{2} \pm \frac{3}{N}, \ldots, \frac{1}{2}\pm (\frac{1}{2}-\frac{1}{N}) \}$ for $\frac{N}{2}$ odd (in addition to $\gamma =0$ and $\gamma =1$). Hence, for $\gamma \in \Gamma_1 \cup \Gamma_2$ and $M= \frac{N}{2}$ 
 we have
 $$\frac{1}{\sqrt{\gamma (1-\gamma) M} } \to 0\quad \text{for } M \to \infty.  $$
 Hence, for $M=\frac{N}{2}$ large and $\gamma \in \Gamma_1 \cup \Gamma_2$, we have
$$\binom{M}{\gamma M}\le  C  \exp(-M (\gamma \log \gamma + (1-\gamma)\log(1-\gamma))).$$
Here,  the above estimate is also trivially true for $\gamma=0$ and $\gamma=1$, where we set $0 \log 0 =0$.
Applying this to the above binomial coefficients $\binom{\frac N2}{\frac N4 (1+m_1)}$ resp.$\binom{\frac N2}{\frac N4 (1+m_2)}$  yields

\begin{align*}
& \P((\mathcal{E}_N^{*,+})^c)\le \sum_{m_1, m_2 \atop  m_1 m_2 \ge 0} C \exp\left(-\frac N2\left( \frac{1+m_1}2 \log \left(\frac{1+m_1}2\right)+\frac{1-m_1}2 \log \left(\frac{1-m_1}2\right)\right)\right)
\\
& \hskip0.5cm\times \exp\left(-\frac{N}{2}\left( \frac{1+m_2}2 \log \left( \frac{1+m_2}2\right)+\frac{1-m_2}2 \log \left(\frac{1-m_2}2\right)\right)\right)
\\
& \hskip0.5cm \times \left\{ \exp \left(-\frac{N^2}8 (m_1^2+m_2^2+2)  I_p(p_N (1+\gamma_N))\right)+
\exp \left(-\frac{N^2}8 (m_1^2+m_2^2+2)  I_p(p_N (1-\gamma_N))\right)
\right.
\\
&  \hskip0.5cm+ \left. \exp \left(-\frac{N^2}4 (m_1m_2+1)  I_q(q_N (1+\kappa_N))\right)+
\exp \left(-\frac{N^2}4 (m_1m_2+1)I_q(q_N (1-\kappa_N))\right)
\right\}
\end{align*}
\begin{align*}
=& \sum_{m_1, m_2 \atop  m_1 m_2 \ge 0} C 2^N \exp\left(-\frac N2\left( \frac{1+m_1}2 \log (1+m_1)+\frac{1-m_1}2 \log (1-m_1)\right)\right)
\\
&\hskip-0cm \times \exp\left(-\frac{N}{2}\left( \frac{1+m_2}2 \log( 1+m_2)+\frac{1-m_2}2 \log (1-m_2)\right)\right)
\\
&  \times \left\{ \exp \left(-\frac{N^2}8 (m_1^2+m_2^2+2)  I_p(p_N (1+\gamma_N))\right)+
\exp \left(-\frac{N^2}8 (m_1^2+m_2^2+2)  I_p(p_N (1-\gamma_N))\right)
\right.
\\
& + \left. \exp \left(-\frac{N^2}4 (m_1m_2+1)  I_q(q_N (1+\kappa_N))\right)+
\exp \left(-\frac{N^2}4 (m_1m_2+1)I_q(q_N (1-\kappa_N))\right)
\right\}.
\end{align*}

We now may separate
the terms that do not depend on the vector $m$. This gives the bound
\begin{eqnarray*}
&&\exp(N\log 2)\exp(-\frac{N^2}{4}I_0)
\\
&&\sum_{m_1, m_2: \atop m_1 m_2 \ge 0 } \exp\left(-\frac N2\left( \frac{1+m_1}2 \log (1+m_1)+\frac{1-m_1}2 \log (1-m_1)\right)\right)\\
&&\hskip-0cm \times \exp\left(-\frac N2 \left( \frac{1+m_2}2 \log (1+m_2)+\frac{1-m_2}2 \log (1-m_2)\right)\right)\\
&& \times \left\{ \exp \left(-\frac{N^2}8 (m_1^2+m_2^2)  I_p(p_N (1+\gamma_N))\right)+
\exp \left(-\frac{N^2}8 (m_1^2+m_2^2)  I_p(p_N (1-\gamma_N))\right)
\right.\\
&& \hskip-.0cm + \left. \exp \left(-\frac{N^2}4 (m_1m_2)  I_q(q_N (1+\kappa_N))\right)+
\exp \left(-\frac{N^2}4 (m_1m_2)I_q(q_N (1-\kappa_N))\right)
\right\},
\end{eqnarray*}
where we have set
$$
I_0:=I_0(N):= \min\{I_p(p_N (1+\gamma_N)), I_p(p_N (1-\gamma_N)), I_q(q_N (1+\kappa_N)), I_q(q_N (1-\kappa_N)).
$$
Obviously all the terms in the sum are bounded by 1, such that the entire sum is at most
\begin{eqnarray}\label{polyI}
&&\sum_{m_1, m_2 : \atop m_1 m_2 \ge 0 } \exp\left(-\frac N2\left( \frac{1+m_1}2 \log (1+m_1)+\frac{1-m_1}2 \log (1-m_1)\right)\right)\nonumber\\
&&\hskip-0cm \times \exp\left(-\frac N2 \left( \frac{1+m_2}2 \log (1+m_2)+\frac{1-m_2}2 \log (1-m_2)\right)\right)\nonumber \\
&& \times \left\{ \exp \left(-\frac{N^2}8 (m_1^2+m_2^2)  I_p(p_N (1+\gamma_N))\right)+
\exp \left(-\frac{N^2}8 (m_1^2+m_2^2)  I_p(p_N (1-\gamma_N))\right)
\right.\nonumber\\
&& \hskip-.0cm \left. + \exp \left(-\frac{N^2}4 (m_1m_2)  I_q(q_N (1+\kappa_N))\right)+
\exp \left(-\frac{N^2}4 (m_1m_2)I_q(q_N (1-\kappa_N))\right)
\right\}\nonumber \\
&\le& 4 \#\{(m_1,m_2): m_1 m_2 \ge 0 \}= 2 \left(\frac N2+1\right)^2.
\end{eqnarray}
It thus remains to show that $\exp(N\log 2)\exp(-\frac{N^2}{4}I_0)$ is exponentially small in $N$.
Computing the terms contributing to $I_0$ we see that
\begin{align*}
I_p(p_N (1+\gamma_N)) &= p_N (1+\gamma_N)\log (1+\gamma_N)+(1-p_N)\left(1-\frac {p_N \gamma_N}{1-p_N}\right)\log \left(1-\frac {p_N \gamma_N}{1-p_N}\right)
\\
I_p(p_N (1-\gamma_N)) &= p_N (1-\gamma_N)\log (1-\gamma_N)+(1-p_N)\left(1+\frac {p_N \gamma_N}{1-p_N}\right)\log \left(1+\frac {p_N \gamma_N}{1-p_N}\right)
\end{align*}
as well as
\begin{align*}
&I_q(q_N (1+\kappa_N)) = q_N (1+\kappa_N)\log (1+\kappa_N)+(1-q_N)\left(1-\frac {q_N \kappa_N}{1-q_N}\right)\log \left(1-\frac {q_N \kappa_N}{1-q_N}\right)\\
&I_q(q_N (1-\kappa_N)) = q_N (1-\kappa_N)\log (1-\kappa_N)+(1-q_N)\left(1+\frac {q_N \kappa_N}{1-q_N}\right)\log \left(1+\frac {q_N \kappa_N}{1-q_N}\right) .
\end{align*}
By Taylor expansion of the $\log$ function we have
$$
(1+x) \log (1+x) \ge x+\frac{x^2}2(1-\frac x 3) \quad \mbox{and  } (1-x) \log (1-x) \ge -x
$$
as well as
$$
(1+x) \log (1+x) \ge x \quad \mbox{and  } (1-x) \log (1-x) \ge -x +\frac {x^2}2
$$
for $0 \le x <1$.

Using these inequalities to estimate $I_p(p_N (1+\gamma_N))$ and $I_p(p_N (1-\gamma_N))$, respectively, we obtain
\begin{equation*} 
I_p(p_N (1+\gamma_N))  \ge  p_N \frac{\gamma_N^2}2(1-\frac{\gamma_N}3)\ge p_N \frac{\gamma_N^2}3
\end{equation*}
(for $N$ large enough) and
\begin{equation*}
I_p(p_N (1-\gamma_N)) \ge  (1-p_N)\frac{p_N \gamma_N}{1-p_N}+ p_N(-\gamma_N+\frac{\gamma_N^2}2)=p_N\frac{\gamma_N^2}2.
\end{equation*}
Similarly,
\begin{equation*}
I_q(q_N (1+\kappa_N)) \ge  q_N \frac{\kappa_N^2}2(1-\frac{\kappa_N}3)\ge q_N \frac{\kappa_N^2}3
\end{equation*}
(for $N$ large enough) and
\begin{equation*}
I_q(q_N (1-\kappa_N)) \ge  (1-q_N)\frac{q_N \kappa_N}{1-q_N}+ q_N(-\kappa+\frac{\kappa_N^2}2)=q_N\frac{\kappa_N^2}2.
\end{equation*}
Therefore
$$
-\frac N4 I_0 \le \max(-N p_N\frac{\gamma^2_N}{12}, -N q_N\frac{\kappa^2_N}{12}).
$$
Thus if we choose $c=d$ such that $c^2 >12 \log 2$ we see that
$$
\log 2 -\frac N4 I_0 <0.
$$
Therefore we obtain that indeed we can find constants $C_1, C_2 >0$ such that
$$
\P((\mathcal{E}_N^{*,+})^c) \le 4 C_1  N^2 \exp(- C_2 N).
$$

We now turn to estimating the second sum in \eqref{sum1}:
Applying the exponential estimates for the sets $C'_N(\sigma)$ and $D'_N(\sigma)$ together with
\eqref{L3} we obtain

\begin{align*}
&\P((\mathcal{E}_N^{*,-})^c)\le \sum_{\sigma \atop m_1 m_2 < 0} \P(A_N(\sigma))+\P(B_N(\sigma))+\P(C'_N(\sigma)) +\P(D'_N(\sigma))\\
&= \sum_{m_1, m_2: \atop  m_1 m_2 <0}\, \sum_{\sigma: m_1(\sigma)=m_1, m_2(\sigma)=m_2} \P(A_N(\sigma))+\P(B_N(\sigma))+\P(C'_N(\sigma)) + \P(D'_N(\sigma))
\\
&= \sum_{m_1, m_2:  \atop m_1 m_2 < 0} \binom{\frac N2}{\frac N4 (1+m_1)}\binom{\frac N2}{\frac N4 (1+m_2)}\times\\
&\times \left\{ \exp \left(-\frac{N^2}8 (m_1^2+m_2^2+2)  I_p(p_N (1+\gamma_N))\right)+
\exp \left(-\frac{N^2}8 (m_1^2+m_2^2+2)  I_p(p_N (1-\gamma_N))\right)
\right.\\
& \, + \left. \exp \left(-\frac{N^2}4 (1-m_1m_2)  I_q(q_N (1+\kappa_N))\right)+
\exp \left(-\frac{N^2}4 (1-m_1m_2)I_q(q_N (1-\kappa_N))\right)
\right\}\\
&\leq\exp(N\log 2)\exp(-\frac{N^2}{4}I_0)
\sum_{m_1, m_2: \atop m_1 m_2 \le 0 } \exp\left(-\frac N2\left( \frac{1+m_1}2 \log (1+m_1)+\frac{1-m_1}2 \log (1-m_1)\right)\right)\\
&\hskip-0cm \times \exp\left(-\frac N2 \left( \frac{1+m_2}2 \log (1+m_2)+\frac{1-m_2}2 \log (1-m_2)\right)\right)\\
& \times \left\{ \exp \left(-\frac{N^2}8 (m_1^2+m_2^2)  I_p(p_N (1+\gamma_N))\right)+
\exp \left(-\frac{N^2}8 (m_1^2+m_2^2)  I_p(p_N (1-\gamma_N))\right)
\right.\\
& \hskip-.0cm + \left. \exp \left(\frac{N^2}4 (m_1m_2)  I_q(q_N (1+\kappa_N))\right)+
\exp \left(\frac{N^2}4 (m_1m_2)I_q(q_N (1-\kappa_N))\right)
\right\},
\end{align*}
Now as in \eqref{polyI}
\begin{align}\label{polyII}
&\sum_{m_1, m_2: \atop m_1 m_2 \le 0 } \exp\left(-\frac N2\left( \frac{1+m_1}2 \log (1+m_1)+\frac{1-m_1}2 \log (1-m_1)\right)\right)\nonumber\\
&\hskip-0cm \times \exp\left(-\frac N2 \left( \frac{1+m_2}2 \log (1+m_2)+\frac{1-m_2}2 \log (1-m_2)\right)\right)\nonumber\\
& \times \left\{ \exp \left(-\frac{N^2}8 (m_1^2+m_2^2)  I_p(p_N (1+\gamma_N))\right)+
\exp \left(-\frac{N^2}8 (m_1^2+m_2^2)  I_p(p_N (1-\gamma_N))\right)
\right.\nonumber\\
& \hskip-.0cm + \left. \exp \left(\frac{N^2}4 (m_1m_2)  I_q(q_N (1+\kappa_N))\right)+
\exp \left(\frac{N^2}4 (m_1m_2)I_q(q_N (1-\kappa_N))\right)
\right\}\nonumber\\
&\le  4 \#\{(m_1,m_2): m_1 m_2 \le 0 \}= 2 \left(\frac N2+1\right)^2.
\end{align}
Indeed, there is a one to one correspondence between the sums considered in \eqref{polyI} and \eqref{polyII}. By flipping all the spins in one of the blocks in \eqref{polyII} we get one summand in \eqref{polyI}, and by this we at the same time change the contribution of the $e^{\frac{N^2}4 (m_1m_2)I_q(\ldots)}$-terms in \eqref{polyII} to $e^{-\frac{N^2}4 (m_1m_2)I_q(\ldots)}$ as they appear \eqref{polyI}. The rest of the terms remains unaltered. It thus remains to show that $\exp(N\log 2)\exp(-\frac{N^2}{4}I_0)$ is exponentially small in $N$, but these terms do not depend on $m$ and we already showed in the first step of the proof that we can make this term exponentially small by choosing $\gamma_N$ and $\kappa_N$ large enough.
This proves the assertion.
\end{proof}

Proposition \ref{prop_full_measure} now follows immediately:
\begin{proof}[Proof of Proposition \ref{prop_full_measure}]:
Just apply the Borel-Cantelli Lemma. The previous lemma states that
$$
\P((\mathcal{E}_N^*)^c)\le 4 C_1  N^2 \exp(- C_2 N).
$$
for some constants $C_1, C_2>0$. The right hand side is summable, hence $(\mathcal{E}^*)^c$ has probability 0.
\end{proof}

We now start with the proof of Theorem \ref{theo_stat_mech}.
Consider the Hamiltonian of the block spin Curie-Weiss model treated in \cite{BRS17_blockmodel} defined in \eqref{standard_hamil} and take $\lambda=\lambda_N= \alpha \frac{q_N}{p_N}$ in place of the $\alpha$ in \eqref{standard_hamil}, i.e.~we consider
$$
\tilde H_{N,\lambda, \beta}(\sigma):= -\frac \beta {2N} \sum_{i \sim j} \sigma_i \sigma_j -\frac \lambda {2N} \sum_{i \not\sim j} \sigma_i \sigma_j, \quad \sigma \in \{-1,+1\}^N
$$
and the corresponding Gibbs measure. Note that $\lambda$ may depend on $N$ and that according to our assumptions we always have that $0\le | \lambda| < \beta$ if $N$ is large enough. The results obtained on the statistical mechanics of this model are stated in Theorem \ref{theo1} below. There $\lambda$ is fixed, but they automatically generalize to the situation with $\lambda_N$ converging to some fixed value. This fixed value in our case is, of course given by $\alpha a$. We will write
$$
H_{N,\alpha,\beta,\vep,\delta,S}(\sigma)=: \tilde H_{N,\alpha a, \beta}(\sigma)-\overline H_N(\sigma).
$$
For $\sigma$ with $m_1(\sigma) m_2(\sigma) \ge 0$ we can make use of \eqref{rewrite} and a similar way to rewrite the Hamiltonian $\tilde H_{N,\alpha a, \beta, S}(\sigma)$ (which is obtained by simply setting all $\vep_{ij}$ and $\delta_{ij}$ to 1 and changing the pre-factor in front of the second term) to obtain
\begin{eqnarray*}
|\overline H_N(\sigma)|&=& | H_{N,\alpha,\beta,\vep,\delta,S}(\sigma)- \tilde H_{N,\alpha a, \beta}(\sigma)|\\
&=& \left|-\frac \beta {2Np} \left[\sum_{i \sim j} 2\vep_{ij} \mathbbm{1}_{i,j \in L^+_b}- \sum_{i \sim j} \vep_{ij}\right]-\frac \alpha {2Np}\left[
\sum_{i \not\sim j} 2\delta_{ij} \mathbbm{1}_{i,j \in L^+_{nb}}- \sum_{i \not \sim j} \delta_{ij}\right]\right.\\
&&+ \left. \frac \beta {2N} \left[\sum_{i \sim j}  2\mathbbm{1}_{i,j \in L^+_b}- \frac{N^2}2\right]
+\frac{\alpha a}{2N}\left[\sum_{i \not\sim j} 2 \mathbbm{1}_{i,j \in L^+_{nb}}- \frac{N^2}2 \right]\right|\\
&\le & \left| \frac \beta {2Np}\sum_{i \sim j}2 \vep_{ij} \mathbbm{1}_{i,j \in L^+_b}-\frac{\beta}{N}|L_b^+|\right|+ \left| \frac \beta {2Np}\sum_{i \sim j} \vep_{ij}-\frac{\beta N}4\right|\nonumber \\
&&+\left| \frac \alpha {2Np}\sum_{i \sim j}2 \delta_{ij} \mathbbm{1}_{i,j \in L^+_{nb}}-\frac{ \alpha a}{N}|L_{nb}^+|\right|+ \left| \frac \alpha {2Np}\sum_{i \not \sim j} \delta_{ij}
-\frac{\alpha a N}4\right|.\nonumber
\end{eqnarray*}
From Proposition \ref{prop_full_measure} we obtain that for all $(\vep, \delta) \in \mathcal{E}^*$, all $N \in \N$ sufficiently large, and all $\sigma \in \Sigma_N$ we have that
\begin{eqnarray}\label{diff1}
\left| \frac \beta {2Np}\sum_{i \sim j}2 \vep_{ij} \mathbbm{1}_{i,j \in L^+_b}-\frac{\beta}{N}|L_b^+|\right| \le \beta \gamma_N \frac{|L_b^+|}{N} \le \beta \gamma_N \frac{N}2
\end{eqnarray}
by just multiplying the defining property for $\mathcal{E}^*_{b,N}$ by $\frac{\beta}{Np}$ and crudly estimating $|L_b^+|$ by $\frac{N^2}2$. In the same way
\begin{equation*}
\left| \frac \beta {2Np}\sum_{i \sim j} \vep_{ij}-\frac{\beta N}4\right| \le \beta \gamma_N \frac{N}2
\end{equation*}
which we get from \eqref{diff1} by considering the configuration $\sigma_i \equiv 1$.

Applying the same trick to the defining property of $\mathcal{E}^*_{nb^+,N}$ we obtain for $N$ large enough and all $\sigma$
\begin{eqnarray*}
\left| \frac \alpha {2Np}\sum_{i \sim j}2 \delta_{ij} \mathbbm{1}_{i,j \in L^+_{nb}}-\frac{ \alpha q_N}{p_N N}|L_{nb}^+|\right| \le \alpha \kappa_N \frac{q_N}{p_N} \frac{|L_{nb}^+|}{N}.
\end{eqnarray*}
By assumption $\frac{q_N}{p_N} \to a$ such that for any $\vep>0$ and $N$ large enough
$\frac{q_N}{p_N} \in (a-\vep, a+\vep)$. Thus
\begin{equation*}
\left| \frac \alpha {2Np}\sum_{i \sim j}2 \delta_{ij} \mathbbm{1}_{i,j \in L^+_{nb}}-\frac{ \alpha a}{N}|L_{nb}^+|\right| \le \alpha a \kappa_N  N
\end{equation*}
as well as
\begin{equation*} 
\left| \frac \alpha {2Np}\sum_{i \not \sim j} \delta_{ij}
-\frac{\alpha a N}4\right| \le \alpha a \kappa_N  N.
\end{equation*}
These estimates together yield
\begin{eqnarray}\label{diff}
|\overline H_N(\sigma)|
\le \beta \gamma_N N+ 2 a \alpha \kappa_N N
\end{eqnarray}
for $N$ large enough and all $\sigma \in \Sigma_N$.

If $\sigma \in \Sigma_N$ satisfies $m_1(\sigma) m_2(\sigma)<0$, we use \eqref{rewrite2} together with similar computations to again obtain
$$
|\overline H_N(\sigma)| \le  \beta \gamma_N N+ 2 a \alpha \kappa_N N.
$$

We will show in the rest of this section, how the bound on $|\overline H_N(\sigma)|$ allows to transfer the Law of Large Numbers for $m$ proved in \cite{BRS17_blockmodel} to our situation.
There the authors show
\begin{theorem}cf. \cite[Proposition 4.1]{BRS17_blockmodel}\label{theo1}
Consider the model with Hamiltonian
$$
\tilde H_{N,\lambda, \beta}(\sigma):= -\frac \beta {2N} \sum_{i \sim j} \sigma_i \sigma_j -\frac \lambda {2N} \sum_{i \not\sim j} \sigma_i \sigma_j, \quad \sigma \in \{-1,+1\}^N
$$
and the corresponding Gibbs measure
$$
\tilde \mu_{N, \lambda, \beta} (\sigma)=
\frac{e^{-\tilde H_{N,\lambda, \beta}(\sigma)}}{\tilde Z_{N, \lambda,\beta}}.
$$
 assume that $|\lambda| < \beta$ and denote by $\tilde \rho_{N,\lambda,\beta}$ the distribution of $m$ under the Gibbs measure $\tilde \mu_{N, \lambda, \beta} $. Then
\begin{itemize}
\item If $\beta +|\lambda|  \le 2$, then $\tilde \rho_{N,\lambda,\beta}$ weakly converges to the Dirac measure in $(0,0)$.
\item If $\beta +|\lambda|  > 2$ and $\lambda =0$, then $\tilde \rho_{N,\lambda,\beta}$ weakly converges to the mixture of Dirac measures $\frac 1 4 \sum_{ s_1, s_2 \in \{-,+\}} \delta_{(s_1 m^+(\beta/2), s_2 m^+(\beta/2))}$.
\item If $\beta +|\lambda| > 2$ and $\lambda>0$, then $\tilde \rho_{N,\lambda,\beta}$ weakly converges to the mixture of Dirac measures
$\frac 1 2 (\delta_{(m^+(\frac{\lambda+\beta}2),  m^+(\frac{\lambda+\beta}2))}+\delta_{(-m^+(\frac{\lambda+\beta}2),  -m^+(\frac{\lambda+\beta}2))}) $.
\item If $\beta +|\lambda| > 2$ and $\lambda<0$, then $\tilde \rho_{N,\lambda,\beta}$ weakly converges to the mixture of Dirac measures
$\frac 1 2 (\delta_{(m^+(\frac{\beta-\lambda}2),  -m^+(\frac{\beta-\lambda}2))}+\delta_{(-m^+(\frac{\beta-\lambda}2),  m^+(\frac{\beta-\lambda}2))}) $.
\end{itemize}
\end{theorem}
Of course, we will apply this result with $\lambda=\alpha a$.
We will transfer it to our situation with the help of the following lemma.
\begin{lemma}\label{lemma2}
As in the previous theorem let $\tilde \rho_{N,\alpha a,\beta}$ be the distribution of $m$ under the measure $\tilde \mu_{N, \alpha a, \beta} $ and let $\rho_{N}$ the distribution of $m$ under the measure $\mu_{N} $. Then for all $m=(m_1,m_2)$ and all realizations of the random graph $(\vep, \delta)\in \mathcal{E}^*$ we have that
$$
\exp\left(-2N(\beta \gamma_N + 2\alpha a \kappa_N )\right)\tilde \rho_{N,\alpha a,\beta}(m) \le \rho_{N} (m)
\le  \exp\left(2N(\beta \gamma_N + 2 \alpha a \kappa_N )\right)\tilde \rho_{N,\alpha a,\beta}(m)
$$
\end{lemma}
\begin{proof}
The statement of the lemma follows immediately, if we consider the form of the Gibbs measures \eqref{gibbs1} and \eqref{gibbs2} together with the above estimate on the difference of the Hamiltonians \eqref{diff} which we need to apply to the numerator and denominator in the definition \eqref{gibbs2}.
\end{proof}
The final observation we now need to make in order to finish the proof of Theorem \ref{theo_stat_mech} is that the vector $m$ obeys a principle of large deviations (LDP, for short) under $\tilde \mu_{N,\lambda,\beta}$.
Indeed the following holds:

\begin{theorem}\label{LDP}(see \cite{LSblock1} Theorem 2.1) \\
For every $S\subset \{1, \ldots, N\}$  with $|S|=\frac{N}{2}$ the vector $m$ obeys a principle of large deviations (LDP) under the Gibbs measure $\tilde \mu_{N,\lambda,\beta}$, with speed $N$ and rate function
$$
J_m(x):= \sup_{y \in \R^2} [F_m(y)-J(y)]-[F_m(x)-J(x)].
$$
Here $F_m:\R^2 \to \R$ is defined by
\begin{equation*}
F_m(x) = \frac 12 \left(\beta x_1^2+ \beta  x_2^2+ 2\lambda  x_1  x_2\right).
\end{equation*}
Moreover,
$$
J(x):=\frac{1}{2} I\left(2x_1\right)+ \frac{1}{2} I\left(2 x_2\right)
$$
for $x \in \R^2$.
Here
$$
I(x):= \frac 12 (1+x) \log (1+x) + \frac 12 (1-x) \log (1-x).
$$

This implies that the convergence in Theorem \ref{theo1} (for $0\le |\lambda|\le \beta$) is exponentially fast.
\end{theorem}
\begin{proof}
In \cite{LSblock1} we give a full proof of this result. Therefore, we just sketch the proof here.
The main idea is to first prove an LDP for $m$ under the uniform distribution on $\sigma \in \Sigma_N$. This is not very difficult. One way to obtain it is to compute logarithmic moment generating function and apply the G\"artner-Ellis Theorem \cite[Theorem 2.3.6]{dembozeitouni}. Once the LDP for $m$ under the uniform measure is established, the theorem follows immediately from the exponential form of the Gibbs measure $\tilde \mu_{N,\lambda,\beta}$ (see \eqref{gibbs1}), the fact that $\tilde H_{N,\lambda,\beta}$ can be expressed as a continuous and bounded function of the vector $m$ (see \eqref{hamil_new}), and  the LDP for integrals of exponential functions (see e.g. \cite[Theorem III.17]{den_hollander_large_deviations} -- a direct consequence of Varadhan's Lemma \cite[Theorem 4.3.1]{dembozeitouni}.
\end{proof}
Lemma \ref{lemma2} and Theorem \ref{LDP} imply an LDP for $m$ also under the measure $\mu_{N,\alpha, \beta, \vep,\delta}$.
\begin{corollary} \label{LDP2}
For every $S\subset \{1, \ldots, N\}$  with $|S|=\frac{N}{2}$ and every realization of the disorder
 $(\vep, \delta)\in \mathcal{E}^*$ the vector $m$ obeys an LDP under
$ \mu_{N,\alpha, \beta, \vep,\delta}$, with speed $N$ and rate function $J^a_m(x)$. Here $J_m^a(x)$ is defined as $J_m(x)$ in Theorem \ref{LDP}, where one replaces $\lambda$ in the definition of $F_m$ by $\alpha a$.
\end{corollary}
\begin{proof}
From Lemma \ref{lemma2} for a closed subset $A \subseteq \R^2$ we have that
\begin{eqnarray*}
\limsup_N \frac 1 N \log \rho_{N,\alpha,\beta, \vep,\delta}(A) &\le& \limsup_N \frac 1 N \log (\tilde \rho_{N, a \alpha,\beta}(A)e^{2N(\beta \gamma_N+2a \alpha \kappa_N)})\\
&\le& \inf_{x \in A} J_m^a(x)+ \limsup_N 2(\beta \gamma_N+2a \alpha \kappa_N)
\end{eqnarray*}
and the second summand on the right hand side is 0 according to the assumptions on $\gamma_N$ and
$\kappa_N$. The lower bound for open sets is obtained analogously.
\end{proof}

\begin{proof}[Proof of Theorem \ref{theo_stat_mech}]
Together with the previous results from this section, a decisive observation is now the following: While in \cite{LSblock1} -- in view of a generalization of the model -- we were striving for a computation of the zeros of the rate function $J_m$ without using Theorem \ref{theo1}, with Theorem \ref{theo1} at hand these zeros can be immediately read off: They are exactly the limit points given in Theorem \ref{theo1}. This follows from the general fact that a LDP always implies a Law of Large Numbers: This Law of Large Numbers may be weak (whether there is a Strong Law of Large Numbers depends on the speed in the LDP) and it may be a generalized Law of Large Numbers in the sense that there is more than one limit point. However, the limit points of this generalized Law of Large Numbers are always the zeros of the rate function.

Now since we know the zeros of the rate function in the LDP for $\tilde \rho_{N, a\alpha,\beta}$ we also know the zeros of the rate function in the LDP for $\rho_{N,\alpha,\beta, \vep,\delta}$, because they are the same. But this means that $m$ converges under $\mu_{N,\alpha,\beta, \vep,\delta}$ to the same limit points as under $\tilde \mu_{N,a \alpha, \beta}$ provided that $(\vep,\delta)\in \mathcal{E}^*$ and $a= \lim_{N\to \infty} \frac{q_N}{p_N}$. But this is the statement of Theorem \ref{theo_stat_mech}.

\end{proof}

Finally, we show that our above results allow to approximate the partition function and to compute the limiting free energy per agent.
Indeed, Proposition \ref{prop_full_measure} allows to approximate $Z_{N,\alpha,\beta,S}$ for a large subset of the realizations of the disorder.
We prove
\begin{lemma}\label{ZNestim}
For any fixed disorder $(\vep,\delta)\in \mathcal{E}^*$, i.e.~with probability one, under the condition that $\lim_{N \to \infty}\alpha \frac{q_n}{p_N} = a$, the partition function $Z_{N,\alpha,\beta, \vep,\delta,S}$  can be approximated by the partition function $\tilde Z_{N,a,\beta}$ in the following way:
$$
\exp\left(-\beta \gamma_N N- 2 a \alpha  \kappa_N N\right)\tilde Z_{N,a\alpha ,\beta}\le Z_{N,\alpha,\beta, \vep,\delta,S}\le \exp\left(\beta \gamma_N N+ 2 a\alpha  \kappa_N N\right)\tilde Z_{N,a \alpha,\beta}.
$$
\end{lemma}
\begin{proof}
The estimate is an immediate consequence of the estimated uniform difference between the Hamiltonians $H$ and $\tilde H$ on $\mathcal{E}^*$, i.e.~\eqref{diff}.
\end{proof}
As an immediate consequence we obtain that for all configurations $(\vep,\delta)\in \mathcal{E}^*$ the free energy of our model exists and equals the free energy of the model treated in \cite{BRS17_blockmodel}.

\begin{corollary}
In our model, for each fixed disorder $(\vep,\delta)\in \mathcal{E}^*$, i.e.~with probability one, the free energy
$$
f:=\lim_{N\to \infty}\frac 1N \log Z_{N,\alpha,\beta, \vep,\delta,S}
$$
exists and satisfies
$$ f= \tilde f :=  \lim_{N\to \infty}\frac 1N \log \tilde Z_{N,a \alpha,\beta}.
$$
\end{corollary}
\begin{proof}
This is obvious from Lemma \ref{ZNestim} and the fact that $\gamma_N$ and $\kappa_N$ converge to $0$, as $N \to \infty$.
\end{proof}

\bibliographystyle{abbrv}
\bibliography{LiteraturDatenbank}
\end{document}